\documentclass[twoside]{article}
\usepackage[cp1251]{inputenc}
\usepackage[T2A]{fontenc}
\usepackage{array}
\usepackage{amssymb}
\usepackage{amsmath}
\usepackage{amsthm}
\usepackage{latexsym}
\usepackage{bm}
\usepackage{enumerate}
\usepackage[dvips]{graphicx}
\usepackage{epsf}
\usepackage{wrapfig}
\usepackage{euscript}
\usepackage[english,russian]{babel}
\usepackage{textcomp}
\newtheorem{theorem}{Theorem}
\newtheorem{lemma}{Lemma}
\newtheorem{definition}{Definition}

\newtheorem{example}{Example}
\newtheorem{corollary}{Corollary}
\begin{document}
{\selectlanguage{english}
\binoppenalty = 10000 %
\relpenalty   = 10000 %

\pagestyle{headings} \makeatletter
\renewcommand{\@evenhead}{\raisebox{0pt}[\headheight][0pt]{\vbox{\hbox to\textwidth{\thepage\hfill \strut {\small Grigory. K. Olkhovikov}}\hrule}}}
\renewcommand{\@oddhead}{\raisebox{0pt}[\headheight][0pt]{\vbox{\hbox to\textwidth{{Model-theoretic characterization of intuitionistic propositional formulas}\hfill \strut\thepage}\hrule}}}
\makeatother

\title{Model-theoretic characterization of intuitionistic propositional formulas\footnote{Acknowledgements:
 I am grateful to Prof. Grigori Mints (Stanford University), who informed me about the open problem which is solved in
 this paper, and spent a lot of his time discussing the possible solutions with me. I am grateful to
 Prof. Balder ten Cate (University of California, Santa Cruz), who made many valuable suggestions on improving the presentation of this result and
 developing its formulation, both in person and in e-mail correspondence.
 I would like to thank all the participants of seminar on logical methods at Stanford
 University for their attention to my talk based on an earlier version of this paper, and for the discussion that
 followed.}}
\author{Grigory K. Olkhovikov\\ Chair of Ontology and Cognition Theory\\ Ural Federal University
\\Fulbright Visiting Scholar at the Philosophy Dept,\\Stanford University \\Bldg 90, Stanford, CA, USA\\
tel.: +1-650-739-6773, +7-922-6173182\\
email: grigory.olkhovikov@usu.ru, grigory.olkhovikov@gmail.com}
\date{}
\maketitle
\begin{quote}
{\bf Abstract.} Notions of $k$-asimulation and asimulation are
introduced as asymmetric counterparts to $k$-bisimulation and
bisimulation, respectively. It is proved that a first-order
formula is equivalent to a standard translation of an
intuitionistic propositional formula iff it is invariant with
respect to $k$-asimulations for some $k$, and then that a
first-order formula is equivalent to a standard translation of an
intuitionistic propositional formula iff it is invariant with
respect to asimulations. Finally, it is proved that a first-order
formula is intuitionistically equivalent to a standard translation
of an intuitionistic propositional formula iff it is invariant
with respect to asimulations between intuitionistic models.
\end{quote}

\begin{quote}
{\bf Keywords.} model theory, intuitionistic logic, propositional
logic, bisimulation, Van Benthem's theorem.
\end{quote}

Van Benthem's well-known  modal characterization theorem (Theorem
\ref{L:vB} below) states that a first-order formula is equivalent
to a standard translation of a modal propositional formula iff it
is invariant with respect to bisimulations. There is also a weaker
`parametrized' version of this result stating that a first-order
formula is equivalent to a standard translation of a modal
propositional formula iff this formula is invariant with respect
to $k$-bisimulations for some $k$. Although both results yield a
convenient model-theoretical technique distinguishing `modal'
first-order formulas from `non-modal' ones, Van Benthem's
characterization theorem, unlike its parametrized version, also
isolates a single property defining expressive powers of modal
propositional logic and thus gives us an important insight into
its nature when this logic is viewed as a fragment of first-order
logic.

It is somewhat surprising that results analogous to Van Benthem's
modal characterization theorem and its parametrized version were
not obtained thus far for the intuitionistic propositional logic,
although the view of the latter as a fragment of modal
propositional logic has a long and established tradition dating
back to Tarski-G\"{o}del translation of this logic into $S4$. The
present paper fills this gap.

The layout of the paper is as follows. Starting from some
notational conventions and preliminary remarks in section
\ref{S:Prel}, we then move on to the proof of a `parametrized'
version of model-theoretic characterization of intuitionistic
propositional logic in section \ref{S:Param} and finally prove the
full unparametrized counterpart to Van Benthem's characterization
theorem for intuitionistic propositional logic in section
\ref{S:Main}. From this latter result we derive a characterization
of equivalence of a first-order formula to a standard translation
of intuitionistic formula on the class of intuitionistic models.
This latter result is of special interest, given that, unlike in
the case of modal propositional logic, not every first-order model
can be treated as a model of intuitionistic propositional logic.
Finally, in section \ref{S:final} we sum up and state some
directions for further research.

\section{Preliminaries}\label{S:Prel}

A formula is a formula of classical predicate logic with identity
whose predicate letters are in vocabulary $\Sigma = \{\,R^2,
P_1^1,\ldots P_n^1,\ldots\,\}$. A model is a model of this logic.
We refer to formulas with lower-case Greek letters distinct from
$\alpha$ and $\beta$, and to sets of formulas with upper-case
Greek letters distinct from $\Sigma$. If $\varphi$ is a formula,
then we associate with it the following finite vocabulary
$\Sigma_\varphi \subseteq \Sigma$ such that $\Sigma_\varphi =
\{\,R^2\,\} \cup \{\,P_i \mid P_i \text{ occurs in }\varphi\,\}$.
If $\psi$ is a formula, $\Sigma' \subseteq \Sigma$ and every
predicate letter occurring in $\psi$ is in $\Sigma'$, then we call
$\psi$ a $\Sigma'$-formula.

 We refer to sequence $x_1,\dots, x_n$ of any objects as
$\bar{x}_n$. If all free variables of a formula $\varphi$ coincide
with a variable $x$, we write $\varphi(x)$. If all free variables
of formulas in $\Gamma$ coincide with $x$, we write $\Gamma(x)$.
We refer to the domain of a model $M$ by $D(M)$. A pointed model
is a pair $(M, a)$, where $M$ is a first-order model and $a \in
D(M)$. If $(M, a)$ is a pointed model, we write $M, a \models
\varphi(x)$ and say that $\varphi(x)$ is true at $(M, a)$ iff for
any variable assignment $f$ in $M$ such that $f(x) = a$, we have
$M, f \models \varphi(x)$. It follows from this convention that
truth of a formula $\varphi(x)$ at a pointed model is to some
extent independent from the choice of its only free variable.

An intuitionistic formula is a formula of intuitionistic
propositional logic. We refer to intuitionistic formulas with
letters $i, j, k$, possibly with primes or subscripts.  We assume
a standard Kripke semantics for intuitionistic propositional
logic.

If $x$ is an individual variable in a first-order language, then
by standard $x$-translation of intuitionistic formulas into
formulas we mean the following map $ST$ defined by induction on
the complexity of the corresponding intuitionistic formula. The
induction goes as follows:
\begin{align*}
&ST(p_n, x) = P_n(x);\\
&ST(\bot, x) = (x \neq x);\\
&ST(i \wedge j, x) = ST(i, x) \wedge
ST(j, x);\\
&ST(i \vee j, x) = ST(i, x) \vee
ST(j, x);\\
&ST(i \to j, x) = \forall y(R(x, y) \to (ST(i, y) \to ST(j, y))).
\end{align*}

Standard conditions are imposed on the variables $x, y$.

By degree of a formula we mean the greatest number of nested
quantifiers occurring in it. Degree of a formula $\varphi$ is
denoted by $r(\varphi)$. Its formal definition by induction on the
complexity of $\varphi$ goes as follows:
\begin{align*}
&r(\varphi) = 0 &&\text{for atomic $\varphi$}\\
&r(\neg\varphi) = r(\varphi)\\
&r(\varphi \circ \psi) = max(r(\varphi), r(\psi)) &&\text{for $\circ \in \{\,\wedge, \vee, \to\,\}$}\\
&r(Qx\varphi) = r(\varphi) + 1 &&\text{for $Q \in \{\,\forall,
\exists\,\}$}
\end{align*}

If $\Sigma' \subseteq \Sigma$, $k \in \mathbb{N}$ and $\varphi(x)$
is a $\Sigma'$-formula such that $r(\varphi) \leq k$, then
$\varphi$ is a $(\Sigma', x, k)$-formula.

\section{A parametrized version of the main result}\label{S:Param}

We start with the definition of an `intuitionistic' counterpart of
$k$-bisimulation.
\begin{definition}\label{D:k-asim}
Let $\Sigma' \subseteq \Sigma$, $R^2 \in \Sigma'$, $(M, a)$, $(N,
b)$ be two pointed $\Sigma'$-models. A binary relation
\[
A \subseteq \bigcup_{n > 0}((D(M)^n \times D(N)^n) \cup (D(N)^n
\times D(M)^n)),
\]
is called $\langle (M,a), (N,b)\rangle_k$-asimulation iff
$(a)A(b)$ and for any $\alpha, \beta \in \{\,M, N\,\}$, any
sequence $(\bar{a'}_m, a') \in D(\alpha)^{m+1}$ and any sequence
$(\bar{b'}_m, b') \in D(\beta)^{m+1}$, whenever we have
$(\bar{a'}_m, a')A(\bar{b'}_m, b')$, the following conditions
hold:

\begin{align}
&\forall P \in \Sigma'(\alpha, a' \models P(x) \Rightarrow \beta, b' \models P(x))\label{E:c1}\\
&b'' \in D(\beta) \wedge b'R^\beta b'' \wedge m < k \Rightarrow\notag\\
&\quad\Rightarrow \exists a'' \in D(\alpha)(a'R^\alpha a'' \wedge
(\bar{b'}_m, b', b'')A(\bar{a'}_m, a', a'') \wedge (\bar{a'}_m,
a', a'')A(\bar{b'}_m, b', b''))\label{E:c2}
\end{align}
\end{definition}
\begin{example}\label{k-asim}
Consider two $\{\,R^2,P^1\,\}$-models $M$ and $N$ such that $D(M)
= \{\,a,b,c\,\}$, $R^M =\{\,(a,b), (a,c)\,\}$, $P^M = \{\,c\,\}$,
and $D(N) = \{\,d,e\,\}$, $R^N =\{\,(d,e)\,\}$, $P^N = \{\,d\,\}$.
Then binary relation $A$ such that $(a)A(d)$, $(d,e)A(a,b)$ and
$(a,b)A(d,e)$ is an $\langle (M,a), (N,d)\rangle_k$-asimulation
for every $k \in \mathbb{N}$.
\end{example}
\begin{lemma}\label{L:asim}
Let $\varphi(x) = ST(i, x)$ for some intuitionistic formula $i$,
and let $r(\varphi) = k$. Let $\Sigma_\varphi \subseteq \Sigma'
\subseteq \Sigma$, $(M, a)$, $(N, b)$ be two pointed
$\Sigma'$-models, let $A$ be an $\langle (M,a),
(N,b)\rangle_l$-asimulation. Then
\begin{align*}
&\forall\alpha, \beta \in \{\,M, N\,\}\forall(\bar{a'}_m, a') \in
D(\alpha)^{m + 1}\forall(\bar{b'}_m, b') \in D(\beta)^{m + 1}\\
&\qquad\qquad((\bar{a'}_m, a')A(\bar{b'}_m, b') \wedge m + k \leq
l \wedge \alpha, a' \models \varphi(x) \Rightarrow \beta, b'
\models \varphi(x)).
\end{align*}
\end{lemma}
\begin{proof}
We proceed by induction on the complexity of $i$. In what
follows we will abbreviate the induction hypothesis by IH.

\emph{Basis.}  Let $i = p_n$. Then $\varphi(x) = P_n(x)$ and we
reason as follows:
\begin{align}
&(\bar{a'}_m, a')A(\bar{b'}_m, b')\label{E:0l1} &&\text{(premise)}\\
&\alpha, a' \models P_n(x)\label{E:0l1-1} &&\text{(premise)}\\
&P_n \in \Sigma'\label{E:0l1-2} &&\text{(by $\Sigma_\varphi \subseteq \Sigma'$)}\\
&\forall P \in \Sigma'(\alpha, a' \models P(x) \Rightarrow \beta,
b'
\models P(x))\label{E:0l2} &&\text{(from \eqref{E:0l1} by \eqref{E:c1})}\\
&\alpha, a' \models P_n(x) \Rightarrow \beta, b'
\models P_n(x)\label{E:0l3} &&\text{(from \eqref{E:0l1-2} and \eqref{E:0l2})}\\
&\beta, b' \models P_n(x)\label{E:0l4} &&\text{(from
\eqref{E:0l1-1} and \eqref{E:0l3})}
\end{align}

The case $i = \bot$ is obvious.

\emph{Induction step.}

\emph{Case 1.} Let $i = j \wedge k$. Then $\varphi(x) = ST(j, x)
\wedge ST(k,x)$ and we reason as follows:

\begin{align}
&(\bar{a'}_m, a')A(\bar{b'}_m, b')\label{E:1l1} &&\text{(premise)}\\
&\alpha, a' \models ST(j, x) \wedge ST(k,x)\label{E:1l2} &&\text{(premise)}\\
&m + r(ST(j, x) \wedge ST(k,x)) \leq l\label{E:1l3} &&\text{(premise)}\\
&r(ST(j, x)) \leq r(ST(j, x) \wedge ST(k,x))\label{E:1l4} &&\text{(by df of $r$)}\\
&r(ST(k, x)) \leq r(ST(j, x) \wedge ST(k,x))\label{E:1l5} &&\text{(by df of $r$)}\\
&\alpha, a' \models ST(j, x)\label{E:1l6} &&\text{(from \eqref{E:1l2})}\\
&\alpha, a' \models ST(k, x)\label{E:1l7} &&\text{(from \eqref{E:1l2})}\\
&m + r(ST(j, x)) \leq l\label{E:1l8} &&\text{(from \eqref{E:1l3} and \eqref{E:1l4})}\\
&m + r(ST(k, x)) \leq l\label{E:1l9} &&\text{(from \eqref{E:1l3} and \eqref{E:1l5})}\\
&\beta, b' \models ST(j, x)\label{E:1l10} &&\text{(from \eqref{E:1l1}, \eqref{E:1l6} and \eqref{E:1l8} by IH)}\\
&\beta, b' \models ST(k, x)\label{E:1l11} &&\text{(from \eqref{E:1l1}, \eqref{E:1l7} and \eqref{E:1l9} by IH)}\\
&\beta, b' \models ST(j, x) \wedge ST(k,x)\label{E:1l12}
&&\text{(from \eqref{E:1l10} and \eqref{E:1l11})}
\end{align}

\emph{Case 2.} Let $i = j \vee k$. Then $\varphi(x) = ST(j, x)
\vee ST(k,x)$ and we have then $\alpha, a' \models ST(j, x) \vee
ST(k,x)$. Assume, without a loss of generality, that $\alpha, a'
\models ST(j, x)$. Then we reason as follows:

\begin{align}
&\alpha, a' \models ST(j, x)\label{E:2l0} &&\text{(premise)}\\
&(\bar{a'}_m, a')A(\bar{b'}_m, b')\label{E:2l1} &&\text{(premise)}\\
&m + r(ST(j, x) \vee ST(k,x)) \leq l\label{E:2l2} &&\text{(premise)}\\
&r(ST(j, x)) \leq r(ST(j, x) \vee ST(k,x))\label{E:2l3} &&\text{(by df of $r$)}\\
&m + r(ST(j, x)) \leq l\label{E:2l4} &&\text{(from \eqref{E:2l2} and \eqref{E:2l3})}\\
&\beta, b' \models ST(j, x)\label{E:2l5} &&\text{(from \eqref{E:2l0}, \eqref{E:2l1} and \eqref{E:2l4} by IH)}\\
&\beta, b' \models ST(j, x) \vee ST(k,x)\label{E:2l6}
&&\text{(from \eqref{E:2l5})}
\end{align}

\emph{Case 3.} Let $i = j \to k$. Then
\[
\varphi(x) = \forall y(R(x, y) \to (ST(j, y) \to ST(k, y))).
\]
Let
\[
\alpha, a' \models \forall y(R(x, y) \to (ST(j, y) \to (ST(k,
y))),
\]
and let
\[
\beta, b' \models \exists y(R(x, y) \wedge (ST(j, y) \wedge \neg
ST(k, y))).
\]
This means that we can choose a $b'' \in D(\beta)$ such that
$b'R^\beta b''$ and $\beta, b'' \models ST(j, y) \wedge \neg ST(k,
y)$. We now reason as follows:
\begin{align}
&\beta, b'' \models ST(j, y) \wedge \neg ST(k,
y)\label{E:4l00} &&\text{(by choice of $b''$)}\\
&b'' \in D(\beta) \wedge b'R^\beta b''\label{E:4l0} &&\text{(by choice of $b''$)}\\
&(\bar{a'}_m, a')A(\bar{b'}_m, b')\label{E:4l1} &&\text{(premise)}\\
&m + r(\varphi(x)) \leq l\label{E:4l2} &&\text{(premise)}\\
&r(\varphi(x)) \geq 1 \label{E:4l3} &&\text{(by df of $r$)}\\
&m < l\label{E:4l4} &&\text{(from \eqref{E:4l2} and
\eqref{E:4l3})}
\end{align}
\begin{align}
\exists a'' \in D(\alpha)(a'R^\alpha a'' \wedge (\bar{b'}_m, b',
b'')A(\bar{a'}_m, a', a'') \wedge (\bar{a'}_m, a',
a'')A(\bar{b'}_m, b', b''))\label{E:4l5}\\
\text{(from \eqref{E:4l0}, \eqref{E:4l1} and \eqref{E:4l4}
by\eqref{E:c2})}\notag
\end{align}
Now choose an $a''$ for which \eqref{E:4l5} is satisfied; we add
the premises following from our choice of $a''$ and continue our
reasoning as follows:
\begin{align}
&a'' \in D(\alpha) \wedge a'R^\alpha a''\label{E:4l6} &&\text{(by choice of $a''$)}\\
&(\bar{b'}_m, b', b'')A(\bar{a'}_m, a', a'')\label{E:4l7} &&\text{(by choice of $a''$)}\\
&(\bar{a'}_m, a', a'')A(\bar{b'}_m, b', b'')\label{E:4l7-1} &&\text{(by choice of $a''$)}\\
&r(ST(j, y)) \leq r(\varphi(x)) - 1\label{E:4l8} &&\text{(by df of $r$)}\\
&r(ST(k, y)) \leq r(\varphi(x)) - 1\label{E:4l8-1} &&\text{(by df of $r$)}\\
&m + 1 + r(ST(j, y)) \leq l\label{E:4l9} &&\text{(from
\eqref{E:4l2} and \eqref{E:4l8})}\\
&m + 1 + r(ST(k, y)) \leq l\label{E:4l9-1} &&\text{(from
\eqref{E:4l2} and \eqref{E:4l8-1})}\\
&\alpha, a'' \models ST(j, x)\label{E:4l10} &&\text{(from \eqref{E:4l00}, \eqref{E:4l7}, \eqref{E:4l9} by IH)}\\
&\alpha, a'' \models \neg ST(k, x)\label{E:4l10-1} &&\text{(from \eqref{E:4l00}, \eqref{E:4l7-1}, \eqref{E:4l9-1} by IH)}\\
&\alpha, a'' \models ST(j, y) \wedge \neg ST(k, y)\label{E:4l10-2} &&\text{(from \eqref{E:4l10}, \eqref{E:4l10-1})}\\
&\alpha, a' \models \exists y(R(x, y) \wedge (ST(j, y) \wedge \neg
ST(k, y)))\label{E:4l11} &&\text{(from \eqref{E:4l6} and
\eqref{E:4l10-2})}
\end{align}
The last line contradicts our initial assumption that
\[
\alpha, a' \models \forall y(R(x, y) \to (ST(j, y) \to (ST(k,
y))).
\]
\end{proof}

\begin{definition}\label{D:k-inv}
A formula $\varphi(x)$ is invariant with respect to
$k$-asimulations iff for any $\Sigma'$ such that $\Sigma_\varphi
\subseteq \Sigma' \subseteq \Sigma$, any pointed $\Sigma'$-models
$(M, a)$ and $(N, b)$, if there exists an $\langle (M, a),(N,
b)\rangle_k$-asimulation $A$ and $M, a \models \varphi(x)$, then
$N, b \models \varphi(x)$.
\end{definition}

\begin{example}\label{k-asim-inv}
Consider again models $M$ and $N$ and binary relation $A$ from
Example \ref{k-asim}. Formula $\exists y(R(x,y) \wedge P(y))$ is
true at $(M,a)$, but not at $(N,d)$. So, since for every $k \in
\mathbb{N}$ $A$ is an $\langle (M,a), (N,d)\rangle_k$-asimulation,
we get that there is no $k$ such that this formula is invariant
with respect to $k$-asimulations.
\end{example}

\begin{corollary}\label{L:c-k-inv}
If $\varphi(x)$ is a standard $x$-translation of an intuitionistic
formula and $r(\varphi) = k$, then $\varphi(x)$ is invariant with
respect to $k$-asimulations.
\end{corollary}

Corollary \ref{L:c-k-inv} immediately follows from Lemma
\ref{L:asim} setting $\alpha = M$, $\beta = N$, $m = 0$, $l = k$,
$a' = a$ and $b' = b$.

Before we state and prove the parametrized version of our main
result, we need to mention a fact from the classical model theory
of first-order logic.

\begin{lemma}\label{L:fin}
For any finite predicate vocabulary $\Sigma'$, any variable $x$
and any natural $k$ there are, up to logical equivalence, only
finitely many $(\Sigma', x, k)$-formulas.
\end{lemma}

This fact is proved as Lemma 3.4 in \cite[pp. 189--190]{EFT}.

\begin{definition}\label{D:conj}
Let $\varphi(x)$ be a formula. A conjunction of $(\Sigma_\varphi,
x, k)$-formulas $\Psi(x)$ is called a complete $(\varphi, x,
k)$-conjunction iff (1) every conjunct in  $\Psi(x)$ is a standard
$x$-translation of an intuitionistic formula; and (2) there is a
pointed model $(M, a)$ such that $M, a \models \Psi(x) \wedge
\varphi(x)$ and for any $(\Sigma_\varphi, x, k)$-formula
$\psi(x)$, if $\psi(x)$ is a standard $x$-translation of an
intuitionistic formula and $M, a \models \psi(x)$, then $\Psi(x)
\models\psi(x)$.
\end{definition}

\begin{lemma}\label{L:conj-ex}
For any formula $\varphi(x)$, any natural $k$, any $\Sigma'$ such
that $\Sigma_\varphi \subseteq \Sigma' \subseteq \Sigma$ and any
pointed $\Sigma'$-model $(M, a)$ such that $M, a \models
\varphi(x)$ there is a complete $(\varphi,x, k)$-conjunction
$\Psi(x)$ such that $M, a \models \Psi(x) \wedge \varphi(x)$.
\end{lemma}
\begin{proof}
Let $\{\,\psi_1(x)\ldots, \psi_n(x),\ldots\,\}$ be the set of all
$(\Sigma_\varphi, x, k)$-formulas that are standard
$x$-translations of intuitionistic formulas true at $(M, a)$. This
set is non-empty since $ST(\bot \to \bot, x)$ will be true at $(M,
a)$. Due to Lemma \ref{L:fin}, we can choose in this set a
non-empty finite subset $\{\,\psi_{i_1}(x)\ldots,
\psi_{i_n}(x)\,\}$ such that any formula from the bigger set is
logically equivalent to (and hence follows from) a formula in this
subset. Therefore, every formula in the bigger set follows from
$\psi_{i_1}(x) \wedge\ldots \wedge\psi_{i_n}(x)$ and we also have
$M, a \models \psi_{i_1}(x) \wedge\ldots \wedge\psi_{i_n}(x)$,
therefore, $\psi_{i_1}(x) \wedge\ldots \wedge\psi_{i_n}(x)$ is a
complete $(\varphi, x, k)$-conjunction.
\end{proof}

\begin{lemma}\label{L:conj-fin}
For any formula $\varphi(x)$ and any natural $k$ there are, up to
logical equivalence, only finitely many complete $(\varphi, x,
k)$-conjunctions.
\end{lemma}
\begin{proof}
It suffices to observe that for any formula $\varphi(x)$ and any
natural $k$, a complete $(\varphi, x, k)$-conjunction is a
$(\Sigma_\varphi, x, k)$-formula. Our lemma then follows from
Lemma \ref{L:fin}.
\end{proof}

In what follows we adopt the following notation for the fact that
for any variable $x$ all $(\Sigma_\varphi, x, k)$-formulas that
are standard $x$-translations of intuitionistic formulas true at
$(M, a)$, are also true at $(N, b)$:
\[
(M, a) \leq_{\varphi,k} (N, b).
\]

\begin{theorem}\label{L:t1}
Let $r(\varphi(x)) = k$ and let $\varphi(x)$ be invariant with
respect to $k$-asimulations. Then $\varphi(x)$ is equivalent to a
standard $x$-translation of an intuitionistic formula.
\end{theorem}
\begin{proof}
We may assume that both $\varphi(x)$ and $\neg\varphi(x)$ are
satisfiable, since both $\bot$ and $\top$ are obviously invariant
with respect to $k$-asimulations and we have, for example, the
following valid formulas:

\begin{align*}
&\bot \leftrightarrow ST(\bot,x), \top \leftrightarrow ST(\bot \to
\bot, x).
\end{align*}

We  may also assume that there are two complete $(\varphi, x, k +
2)$-conjunctions $\Psi(x), \Psi'(x)$ such that
$\Psi'(x)\models\Psi(x)$, and both formulas $\Psi(x) \wedge
\varphi(x)$ and $\Psi'(x) \wedge \neg\varphi(x)$ are satisfiable.

For suppose otherwise. Then take the set of all complete
$(\varphi, x, k + 2)$-conjunctions $\Psi(x)$ such that the formula
$\Psi(x) \wedge \varphi(x)$ is satisfiable. This set is non-empty,
because $\varphi(x)$ is satisfiable, and by Lemma \ref{L:conj-ex},
it can be satisfied only together with some complete $(\varphi, x,
k + 2)$-conjunction. Now, using Lemma \ref{L:conj-fin}, choose in
it a finite non-empty subset $\{\,\Psi_{i_1}(x)\ldots,
\Psi_{i_n}(x)\,\}$ such that any complete $(\varphi, x, k +
2)$-conjunction is equivalent to an element of this subset. We can
show that $\varphi(x)$ is logically equivalent to
$\Psi_{i_1}(x)\vee\ldots \vee\Psi_{i_n}(x)$. In fact, if $M, a
\models \varphi(x)$ then, by Lemma \ref{L:conj-ex}, at least one
complete $(\varphi, x, k + 2)$-conjunction is true at $(M, a)$ and
therefore, its equivalent in $\{\,\Psi_{i_1}(x)\ldots,
\Psi_{i_n}(x)\,\}$ is also true at $(M, a)$, and so, finally we
have $M, a \models \Psi_{i_1}(x)\vee\ldots \vee\Psi_{i_n}(x)$. In
the other direction, if $M, a \models \Psi_{i_1}(x)\vee\ldots
\vee\Psi_{i_n}(x)$, then for some $1 \leq j \leq n$ we have $M, a
\models \Psi_{i_j}(x)$. Then, since
$\Psi_{i_j}(x)\models\Psi_{i_j}(x)$ and by the choice of
$\Psi_{i_j}(x)$ the formula $\Psi_{i_j}(x) \wedge \varphi(x)$ is
satisfiable, so, by our assumption, the formula $\Psi_{i_j}(x)
\wedge \neg\varphi(x)$ must be unsatisfiable, and hence
$\varphi(x)$ must follow from $\Psi_{i_j}(x)$. But in this case we
will have $M, a \models \varphi(x)$ as well. So $\varphi(x)$ is
logically equivalent to $\Psi_{i_1}(x)\vee\ldots,
\vee\Psi_{i_n}(x)$ but the latter formula, being a disjunction of
conjunctions of standard $x$-translations of intuitionistic
formulas, is itself a standard $x$-translation of an
intuitionistic formula, and so we are done.

If, on the other hand, one can take two complete $(\varphi, x, k +
2)$-conjunctions $\Psi(x), \Psi'(x)$ such that
$\Psi'(x)\models\Psi(x)$, and formulas $\Psi(x) \wedge \varphi(x)$
and $\Psi'(x) \wedge \neg\varphi(x)$ are satisfiable, we reason as
follows. Take a pointed $\Sigma_\varphi$-model $(M,a)$ such that
$M, a \models \Psi(x) \wedge \varphi(x)$ and for any
$(\Sigma_\varphi, x, k + 2)$-formula $\psi(x)$, if $\psi(x)$ is a
standard $x$-translation of an intuitionistic formula true at
$(M,a)$, then $\psi(x)$ follows from $\Psi(x)$, and take any
pointed model $(N, b)$ such that $N, b \models \Psi'(x) \wedge
\neg\varphi(x)$.

We can construct an $\langle (M,a), (N, b)\rangle_k$-asimulation
and thus obtain a contradiction in the following way.

Let $\alpha, \beta \in \{\,M, N\,\}$ and let $(\bar{a'}_m, a')$
and $(\bar{b'}_m,b')$ be in $D(\alpha)^{m+1}$ and
$D(\beta)^{m+1}$, respectively. Then
\[
(\bar{a'}_m,a')A(\bar{b'}_m,b') \Leftrightarrow (m \leq k \wedge
(\alpha, a') \leq_{\varphi,k - m + 2} (\beta, b')).
\]

By choice of $\Psi(x), \Psi'(x)$ and the independence of truth at
a pointed model from the choice of a single free variable in a
formula we obviously have $(a)A(b)$.

Further, since the degree of any atomic formula is $0$, and the
above condition implies that $k - m + 2 \geq 2$, it is evident
that for any $(\bar{a'}_m,a')A(\bar{b'}_m,b')$ and any predicate
letter $P\in\Sigma_\varphi$ we have $\alpha,a' \models P(x)
\Rightarrow \beta, b' \models P(x)$.

To verify condition \eqref{E:c2}, take any
$(\bar{a'}_m,a')A(\bar{b'}_m,b')$ such that $m < k$ and any $b''
\in D(\beta)$ such that $b'R^\beta b''$. In this case we will also
have $m + 1 \leq k$.

Then consider the following two sets:
\begin{align*}
&\Gamma = \{\,ST(i, x) \mid ST(i, x)\text{ is a $(\Sigma_\varphi,
x, k
+ 1 - m)$-formula, and }\beta, b'' \models ST(i, x)\,\};\\
&\Delta = \{\,ST(i, x) \mid ST(i, x)\text{ is a $(\Sigma_\varphi,
x, k + 1 - m)$-formula, and }\beta, b'' \models \neg ST(i, x)\,\}.
\end{align*}
These sets are non-empty, since by our assumption we have $k + 1 -
m \geq 1$. Therefore, as we have $r(ST(\bot, x)) = 0$ and
$r(ST(\bot \to \bot, x)) = 1$, we will also have $ST(\bot, x) \in
\Delta$ and $ST(\bot \to \bot, x) \in \Gamma$. Then, according to
our Lemma \ref{L:fin}, there are finite non-empty sets of logical
equivalents for both $\Gamma$ and $\Delta$. Choosing these finite
sets, we in fact choose some finite $\{\,ST(i_1,x)\ldots ST(i_t,
x)\,\} \subseteq \Gamma$, $\{\,ST(j_1,x)\ldots ST(j_u, x)\,\}
\subseteq \Delta$ such that
\begin{align*}
&\forall \psi(x) \in \Gamma(ST(i_1,x)\wedge\ldots \wedge ST(i_t,
x)
\models \psi(x));\\
&\forall \chi(x) \in \Delta(\chi(x)\models ST(j_1,x)\vee\ldots
\vee ST(j_u, x)).
\end{align*}
But then we obtain that the formula
\[
ST((i_1\wedge\ldots \wedge i_t) \to (j_1\vee\ldots \vee j_u), x)
\]
is false at $(\beta, b')$. In fact, $b''$ disproves this
implication for $(\beta, b')$. But every formula both in
$\{\,ST(i_1,x)\ldots ST(i_t, x)\,\}$ and $\{\,ST(j_1,x)\ldots
ST(j_u, x)\,\}$ is, by their choice, a $(\Sigma_\varphi, x, k + 1
- m)$-formula, and so the implication under consideration must be
a $(\Sigma_\varphi,x,k + 2 - m)$-formula. Note, further, that by
$(\bar{a'}_m,a')A(\bar{b'}_m,b')$ we have
\[ (\alpha, a')
\leq_{\varphi,k - m + 2} (\beta, b')
\]
and therefore this
implication must be false at $(\alpha, a')$ as well. But then take
any $a'' \in D(\alpha)$ such that $a'R^\alpha a''$ and $a''$
verifies the conjunction in the antecedent of the formula but
falsifies its consequent. We must conclude then, by the choice of
$\{\,ST(i_1,x)\ldots ST(i_t, x)\,\}$, that $\alpha, a'' \models
\Gamma$ and so, by the definition of $A$, and given that $m + 1
\leq k$, that $(\bar{b'}_m,b', b'')A(\bar{a'}_m,a', a'')$. Since,
in addition, $a''$ falsifies every formula from
$\{\,ST(j_1,x)\ldots ST(j_u, x)\,\}$, then, by the choice of this
set, we must conclude that every $(\Sigma_\varphi,x,k + 1 -
m)$-formula that is a standard $x$-translation of an
intuitionistic formula false at $(\beta, b'')$ is also false at
$(\alpha, a'')$. But then, again by the definition of $A$, and
given the fact that $m + 1 \leq k$, we must also have
$(\bar{a'}_m,a', a'')A(\bar{b'}_m,b', b'')$, and so condition
\eqref{E:c2} holds.

Therefore $A$ is an $\langle (M,a), (N, b)\rangle_k$-asimulation
and we have got our contradiction in place.
\end{proof}

\begin{theorem}\label{L:param}
A formula $\varphi(x)$ is equivalent to a standard $x$-translation
of an intuitionistic formula iff there exists a $k \in \mathbb{N}$
such that $\varphi(x)$ is invariant with respect to
$k$-asimulations.
\end{theorem}
\begin{proof}
Let $\varphi(x)$ be equivalent to $ST(i,x)$. Then by Corollary
\ref{L:c-k-inv}, $ST(i,x)$ is invariant with respect to
$r(ST(i,x))$-asimulations, and, therefore, so is $\varphi(x)$. In
the other direction, let $\varphi(x)$ be invariant with respect to
$k$-asimulations for some $k$. If $k \leq r(\varphi)$, then every
$r(\varphi)$-asimulation is  $k$-asimulation, so $\varphi(x)$ is
invariant with respect to $r(\varphi)$-asimulations and hence, by
Theorem \ref{L:t1}, $\varphi(x)$ is equivalent to a standard
$x$-translation of an intuitionistic formula. If, on the other
hand, $r(\varphi) < k$, then set $l = k - r(\varphi)$ and consider
variables $\bar{y}_l$ not occurring in $\varphi(x)$. Then
$r(\forall\bar{y}_l\varphi(x)) = k$ and $\varphi(x)$ is logically
equivalent to $\forall\bar{y}_l\varphi(x)$, so the latter formula
is also invariant with respect to $k$-asimulations, and hence by
Theorem \ref{L:t1} $\forall\bar{y}_l\varphi(x)$ is logically
equivalent to a standard $x$-translation of an intuitionistic
formula. But then $\varphi(x)$ is equivalent to this standard
$x$-translation as well.
\end{proof}

\section{The main result}\label{S:Main}

We begin with a definition of an `intuitionistic' counterpart to
bisimulation:
\begin{definition}\label{D:asim}
Let $\Sigma' \subseteq \Sigma$, $R^2 \in \Sigma'$, $(M, a)$, $(N,
b)$ be two pointed $\Sigma'$-models. A binary relation
\[
A \subseteq (D(M) \times D(N)) \cup (D(N)\times D(M)),
\]
is called $\langle (M,a), (N,b)\rangle$-asimulation iff $aAb$ and
for any $\alpha, \beta \in \{\,M, N\,\}$, any $a' \in D(\alpha)$,
$b' \in D(\beta)$ whenever we have $a'Ab'$, the following
conditions hold:

\begin{align}
&\forall P \in \Sigma'(\alpha, a' \models P(x) \Rightarrow \beta, b' \models P(x))\label{E:c11}\\
&b'' \in D(\beta) \wedge b'R^\beta b''\Rightarrow \exists a'' \in
D(\alpha)(a'R^\alpha a'' \wedge b''Aa'' \wedge
a''Ab'')\label{E:c22}
\end{align}
\end{definition}
\begin{example}\label{asim}
Consider again models $M$ and $N$ from Example \ref{k-asim}.
Binary relation $B = \{\,(a,d),(b,e),(e,b)\,\}$ is an $\langle
(M,a), (N,d)\rangle$-asimulation.
\end{example}
\begin{lemma}\label{L:k-asim1}
Let $A$ be an $\langle (M,a), (N,b)\rangle$-asimulation, and let
\[
A' = \{\,\langle(\bar{c}_n,c'),(\bar{d}_n,d')\rangle\mid
c'Ad'\,\}.
\]
Then $A'$ is an $\langle (M,a), (N, b)\rangle_k$-asimulation for
any $k \in \mathbb{N}$.
\end{lemma}
\begin{proof}
We obviously have $(a)A'(b)$, and since for any $\alpha, \beta \in
\{\,M, N\,\}$, and any $(\bar{c}_n,c')$ in $D(\alpha)^{n+1}$,
$(\bar{d}_n,d')$ in $D(\beta)^{n+1}$ such that
$(\bar{c}_n,c')A'(\bar{d}_n,d')$ we have $c'Ad'$, condition
\eqref{E:c1} for $A'$ follows from the fulfilment of condition
\eqref{E:c11} for $A$. Also, if $(\bar{c}_n,c')A'(\bar{d}_n,d')$
then $c'Ad'$, and if, further, $d'' \in D(\beta)$ and $d'R^\beta
d''$ then by condition \eqref{E:c22} we can choose $c'' \in
D(\alpha)$ such that $c'R^\alpha c''$, $c''Ad''$ and $d''Ac''$.
But then, by definition of $A'$ we will also have
$(\bar{c}_n,c',c'')A'(\bar{d}_n,d',d'')$ and
$(\bar{d}_n,d',d'')A'(\bar{c}_n,c',c'')$ so condition \eqref{E:c2}
for $A'$ is fulfilled for every $k$.
\end{proof}
\begin{definition}\label{D:inv}
A formula $\varphi(x)$ is invariant with respect to asimulations
iff for any $\Sigma'$ such that $\Sigma_\varphi \subseteq \Sigma'
\subseteq \Sigma$, any pointed $\Sigma'$-models $(M, a)$ and $(N,
b)$, if there exists an $\langle (M, a),(N, b)\rangle$-asimulation
$A$ and $M, a \models \varphi(x)$, then $N, b \models \varphi(x)$.
\end{definition}
\begin{example}\label{ex-inv}
Consider again models $M$ and $N$ from Example \ref{k-asim}. Since
$\exists y(R(x,y) \wedge P(y))$ is true at $(M,a)$, but not at
$(N,d)$, the fact that binary relation $B$ from Example \ref{asim}
is an $\langle (M,a), (N,d)\rangle$-asimulation means that this
formula is not invariant with respect to asimulations.
\end{example}
\begin{corollary}\label{L:c-inv}
If $\varphi(x)$ is equivalent to a standard $x$-translation of an
intuitionistic formula, then $\varphi(x)$ is invariant with
respect to asimulations.
\end{corollary}
\begin{proof}
Let $\varphi(x)$ be equivalent to a standard $x$-translation of an
intuitionistic formula, let $A$ be an $\langle (M, a),(N,
b)\rangle$-asimulation and let $A'$ be defined as in Lemma
\ref{L:k-asim1}. Then by this Lemma $A'$ is an $\langle (M, a),(N,
b)\rangle_k$-asimulation for every $k$. So if we have $M, a
\models \varphi(x)$, but not $N, b \models \varphi(x)$, then
$\varphi(x)$ is not invariant with respect to $k$-asimulations for
any $k$, which is in contradiction with Theorem \ref{L:param}.
\end{proof}

In what follows we will also need some notions and facts from
model theory of modal propositional logic. Thus, standard modal
$x$-translation $Tr(m, x)$ of a modal propositional formula $m$ in
first-order logic is defined by the following induction on the
complexity of modal propositional formula:
\begin{align*}
&Tr(p_n, x) = P_n(x);\\
&Tr(m \wedge m', x) = Tr(m, x) \wedge
Tr(m', x);\\
&Tr(\neg m, x) = \neg Tr(m, x);\\
&Tr(\Box m, x) = \forall y(R(x, y) \to Tr(m, y)).
\end{align*}

Another important idea is the notion of bisimulation:
\begin{definition}\label{D:bisim}
Let $\Sigma'$ be a predicate vocabulary such that $\Sigma'
\subseteq \Sigma$, $R^2 \in \Sigma'$, and $(M, a)$, $(N, b)$ be
pointed $\Sigma'$-models. Then a binary relation $E \subseteq D(M)
\times D(N)$ is a $\langle (M, a),(N, b)\rangle$-bisimulation iff
$aEb$ and for any $a' \in M$, $b' \in N$, whenever $a'Eb'$, the
following conditions hold:
\begin{align}
&\forall P \in \Sigma'(M, a' \models P(x) \Leftrightarrow N, b' \models P(x));\label{E:co1}\\
&(a'' \in D(M) \wedge a'R^Ma'') \Rightarrow \exists b'' \in
D(N)(b'R^Nb'' \wedge a''Eb'');\label{E:co2}\\
&(b'' \in D(N) \wedge b'R^Nb'') \Rightarrow \exists a'' \in
D(M)(a'R^Ma'' \wedge a''Eb'').\label{E:co3}
\end{align}
\end{definition}

\begin{definition}\label{D:bisim-inv}
A formula $\varphi(x)$ is invariant with respect to bisimulations
iff for any $\Sigma'$ such that $\Sigma_\varphi \subseteq \Sigma'
\subseteq \Sigma$, any pointed $\Sigma'$-models $(M, a)$ and $(N,
b)$, and any $\langle (M, a),(N, b)\rangle$-bisimulation it is
true that
\[
M, a \models \varphi(x) \Rightarrow N, b \models \varphi(x).
\]
\end{definition}
The concept of standard modal translation and that of bisimulation
invariance are tied together by  Van Benthem's famous modal
characterization theorem:
\begin{theorem}\label{L:vB}
A formula $\varphi(x)$ is invariant with respect to bisimulations
iff it is equivalent to a standard modal $x$-translation of a
modal propositional formula.
\end{theorem}
Its proof can be found, for example, in \cite[Theorem 2.68, pp.
103--104]{BRV}. It is easy to see that our main result below
(Theorem \ref{L:final}) is in an analogy with Van Benthem's
characterization theorem for intuitionistic propositional logic
both in its formulation and in methods of proof employed.
\begin{lemma}\label{L:bisim-asim}
Let $\varphi(x)$ be a formula invariant with respect to
asimulations. Then:
\begin{enumerate}
\item $\varphi(x)$ is invariant with respect to bisimulations.
\item $\neg\varphi(x)$ is invariant with respect to bisimulations.
\end{enumerate}
\end{lemma}
\begin{proof}
(1) Let  $\Sigma_\varphi \subseteq \Sigma' \subseteq \Sigma$,
let $(M, a)$, $(N, b)$ be pointed $\Sigma'$-models and let $E$ be
an $\langle (M, a),(N, b)\rangle$-bisimulation such that $M, a
\models \varphi(x)$ but not $N, b \models \varphi(x)$. Then define
$A$ as $E \cup E^{-1}$. It is easy to verify that $A$ is an
$\langle (M, a),(N, b)\rangle$-asimulation: we obviously have
$aAb$, and condition \eqref{E:c11} is fulfilled.

To verify \eqref{E:c22}, assume that $a'Ab'$. Then either $a' \in
D(M) \wedge b' \in D(N)$ or $a' \in D(N) \wedge b' \in D(M)$. So
in the former case, by Definition \ref{D:bisim} and our definition
of $A$, we must have $a'Eb'$, while in the latter case we must
have $b'Ea'$. Therefore, in the former case, if $b'R^Nb''$ we
apply condition \eqref{E:co3} and choose $a'' \in D(M)$ such that
$a'R^Ma'' \wedge a''Eb''$, and so, by definition of $A$, we have
both $a''Ab''$ and $b''Aa''$. In the latter case, if $b'R^Mb''$ we
apply condition \eqref{E:co2} and choose $a'' \in D(N)$ such that
$a'R^Na'' \wedge b''Ea''$, and so, again by definition of $A$, we
have both $b''Aa''$ and $a''Ab''$. Thus $A$ is an $\langle (M,
a),(N, b)\rangle$-asimulation and $\varphi(x)$ is not invariant
with respect to asimulations, contrary to our assumption. The
first statement of the lemma is proved.

(2) Let  $\Sigma_\varphi \subseteq \Sigma' \subseteq \Sigma$, let
$(M, a)$, $(N, b)$ be pointed $\Sigma'$-models and let $E$ be an
$\langle (M, a),(N, b)\rangle$-bisimulation such that $N, b
\models \varphi(x)$ but not $M, a \models \varphi(x)$. Again,
define $A$ as $E \cup E^{-1}$. In the previous paragraph it was
established that $A$ verifies conditions \eqref{E:c11} and
\eqref{E:c22}. But since $aEb$, we also have $bAa$ and so $A$ is
in fact an $\langle (N, b),(M, a)\rangle$-asimulation, which
contradicts our assumption that  $\varphi(x)$ is invariant with
respect to asimulations.
\end{proof}

\begin{definition}\label{D:sat}
A model $M$ is called m-saturated iff for any $a \in D(M)$ and for
any set $\Theta(x)$ of standard modal $x$-translations of modal
propositional formulas it is true that
\begin{align*}
[\forall(\Theta'(x) \subseteq \Theta(x))(\Theta'(x)\text{ is
finite} \Rightarrow \exists b\in D(M)(aR^Mb \wedge M,b\models
\Theta'(x)))] \Rightarrow\\
\Rightarrow \exists c\in D(M)(aR^Mc \wedge M,c\models \Theta(x)).
\end{align*}
\end{definition}
Let $\Sigma' \subseteq \Sigma$. In what follows we adopt the
following notation for the fact that for any $x$ all
$\Sigma'$-formulas that are standard $x$-translations of
intuitionistic formulas true at $(M, a)$, are also true at $(N,
b)$:
\[
(M, a) \leq_{\Sigma'} (N, b).
\]
\begin{lemma}\label{L:sat}
Let $\Sigma' \subseteq \Sigma$, let $M$, $N$ be m-saturated
$\Sigma'$-models and let $(M, a) \leq_{\Sigma'} (N, b)$. Then
relation $\leq_{\Sigma'}$ is an $\langle (M,a),
(N,b)\rangle$-asimulation.
\end{lemma}
\begin{proof}
It is obvious that $aAb$, and since for any unary predicate
letter $P$ and variable $x$ formula $P(x)$ is a standard
$x$-translation of an atomic intuitionistic formula, condition
\eqref{E:c11} is trivially satisfied for $\leq_{\Sigma'}$. To
verify condition \eqref{E:c22}, choose any $\alpha, \beta \in
\{\,M, N\,\}$, and $a' \in D(\alpha)$, $b',b'' \in D(\beta)$ such
that $(\alpha, a') \leq_{\Sigma'} (\beta, b')$ and $b'R^\beta
b''$. Then choose any variable $x$ and consider the following two
sets:
\begin{align*}
&\Gamma = \{\,i \mid ST(i, x)\text{ is a $\Sigma'$-formula, and }\beta, b'', \models ST(i, x)\,\};\\
&\Delta = \{\,i \mid ST(i, x)\text{ is a $\Sigma'$-formula, and
}\beta, b'', \models \neg ST(i, x)\,\}.
\end{align*}
We have by the choice of $\Gamma$, $\Delta$ that for every finite
$\Gamma' \subseteq \Gamma$ and $\Delta' \subseteq \Delta$ the
formula $ST(\bigwedge(\Gamma') \to \bigvee(\Delta'), x)$ is
disproved by $b''$ for $(\beta, b')$. So, by our premise that
$(\alpha, a') \leq_{\Sigma'} (\beta, b')$, the standard
translation of every such implication must be false at $(\alpha,
a')$ as well. This means that every finite subset of the set
\[
\{\,ST(i,x)\mid i\in\Gamma\,\} \cup \{\,\neg ST(i,x)\mid
i\in\Delta\,\}
\]
is true at some $a'' \in D(\alpha)$ such that $a'R^\alpha a''$.
(We set $\Delta' = \{\,ST(\bot, x)\,\}$ if the finite set in
question has an empty intersection with $\Delta$ and $\Gamma' =
\{\,ST(\bot \to \bot, x)\,\}$ if it has an empty intersection with
$\Gamma$.) But by Corollary \ref{L:c-inv} and Lemma
\ref{L:bisim-asim} every formula in the set under consideration is
invariant with respect to bisimulations and hence equivalent to a
standard modal $x$-translation of a modal propositional formula.
Therefore, by m-saturation of both $M$ and $N$ there must be an
$a'' \in D(\alpha)$ such that $a'R^\alpha a''$ and
\[
\alpha, a'' \models \{\,ST(i,x)\mid i\in\Gamma\,\} \cup \{\,\neg
ST(i,x)\mid i\in\Delta\,\}.
\]
By choice of $\Gamma$ and $\Delta$ and by the independence of
truth at a pointed model from the choice of a single free variable
in a formula we will have both $(\alpha, a'') \leq_{\Sigma'}
(\beta, b'')$ and $(\beta, b'') \leq_{\Sigma'} (\alpha, a'')$ and
so condition \eqref{E:c22} is also verified.
\end{proof}
\begin{lemma}\label{L:ext}
Let  $\Sigma_\varphi \subseteq \Sigma' \subseteq \Sigma$ and let
$M$ be a $\Sigma'$-model. Then there is a $\Sigma'$-model $N$ such
that $N$ is an extension of $M$, $N$ is m-saturated and there is a
map $f: D(M) \to D(N)$ such that for any formula $\varphi(x)$
which is invariant with respect to bisimulations and any $a \in
D(M)$ it is true that
\[
M, a \models \varphi(x) \Leftrightarrow N, f(a) \models
\varphi(x).
\]
\end{lemma}
\begin{proof}
Let $N$ be an ultrafilter extension of $M$ and let
 $f(a)$ be the principal ultrafilter generated
by $a$ for any $a \in D(M)$. Then our lemma follows from
Propositions 2.59 and 2.61 of \cite[pp. 96-97]{BRV} and Theorem
\ref{L:vB}.
\end{proof}

We are prepared now to state and prove our main result.

\begin{theorem}\label{L:main}
Let $\varphi(x)$ be invariant with respect to asimulations. Then
$\varphi(x)$ is equivalent to a standard $x$-translation of an
intuitionistic formula.
\end{theorem}
\begin{proof}
We may assume that $\varphi(x)$ is satisfiable, for $\bot$ is
clearly invariant with respect to asimulations and $\bot
\leftrightarrow ST(\bot, x)$ is a valid formula. In what follows
we will write $IC(\varphi(x))$ for the set of
$\Sigma_\varphi$-formulas that are standard $x$-translations of
intuitionistic formulas following from $\varphi(x)$. For any
pointed $\Sigma_\varphi$-model $(M, a)$ we will denote the set of
$\Sigma_\varphi$-formulas that are standard $x$-translations of
intuitionistic formulas true at $(M, a)$, or \emph{intuitionistic
$\Sigma_\varphi$-theory} of $(M, a)$ by $IT_\varphi(M,a)$. It is
obvious that for any pointed $\Sigma_\varphi$-models $(M, a)$ and
$(N, b)$ we will have $(M, a) \leq_{\Sigma_\varphi}(N, b)$ if and
only if $IT_\varphi(M,a) \subseteq IT_\varphi(N,b)$.

Our strategy will be to show that $IC(\varphi(x)) \models
\varphi(x)$. Once this is done we will apply compactness of
first-order logic and conclude that $\varphi(x)$ is equivalent to
a finite conjunction of standard $x$-translations of
intuitionistic formulas and hence to a standard $x$-translation of
the corresponding intuitionistic conjunction.

To show this, take any pointed $\Sigma_\varphi$-model $(M, a)$
such that $M, a \models IC(\varphi(x))$. Such a model exists,
because $\varphi(x)$ is satisfiable and  $IC(\varphi(x))$ will be
satisfied in any pointed model satisfying $\varphi(x)$. Then we
can also choose a pointed $\Sigma_\varphi$-model $(N, b)$ such
that $N, b \models \varphi(x)$ and $IT_\varphi(N, b) \subseteq
IT_\varphi(M,a)$.

For suppose otherwise. Then for any pointed $\Sigma_\varphi$-model
$(N, b)$ such that $N, b \models \varphi(x)$ we can choose an
intuitionistic formula $i_{(N, b)}$ such that $ST(i_{(N, b)}, x)$
is a $\Sigma_\varphi$-formula true at $(N, b)$ but not at $(M,
a)$. Then consider the set
\[
S = \{\,\varphi(x)\,\} \cup \{\,\neg ST(i_{(N, b)}, x)\mid N, b
\models \varphi(x)\,\}
\]
Let $\{\,\varphi(x), \neg ST(i_{(N_1, b_1)}, x)\ldots , \neg
ST(i_{(N_u, b_u)}, x)\,\}$ be a finite subset of this set. If this
set is  unsatisfiable, then we must have $\varphi(x) \models
ST(i_{(N_1, b_1)}, x)\vee\ldots \vee ST(i_{(N_u, b_u)}, x)$, but
then we will also have $(ST(i_{(N_1, b_1)}, x)\vee\ldots \vee
ST(i_{(N_u, b_u)}, x)) \in IC(\varphi(x)) \subseteq
IT_\varphi(M,a)$, and hence $(ST(i_{(N_1, b_1)}, x)\vee\ldots \vee
ST(i_{(N_u, b_u)}, x))$ will be true at $(M, a)$. But then at
least one of $ST(i_{(N_1, b_1)}, x)\ldots ,ST(i_{(N_u, b_u)}, x)$
must also be  true at $(M, a)$, which contradicts the choice of
these formulas. Therefore, every finite subset of $S$ is
satisfiable, and by compactness $S$ itself is satisfiable as well.
But then take any pointed $\Sigma_\varphi$-model $(N',b')$ of $S$
and this will be a model for which we will have both $N', b'
\models ST(i_{(N', b')}, x)$ by choice of $i_{(N', b')}$ and $N',
b' \models \neg ST(i_{(N', b')}, x)$ by the satisfaction of $S$, a
contradiction.

Therefore, we will assume in the following that $(M, a)$, $(N, b)$
are pointed $\Sigma_\varphi$-models, $M, a \models
IC(\varphi(x))$, $N, b \models \varphi(x)$, and $IT_\varphi(N, b)
\subseteq IT_\varphi(M,a)$. Then, according to Lemma \ref{L:ext},
consider m-saturated models $M'$, $N'$ that are extensions of $M$
and $N$, respectively, and maps $f:D(M) \to D(M')$ and $g:D(N) \to
D(N')$ such that for any $\Sigma_\varphi$-formula $\chi(x)$ which
is invariant with respect to bisimulations and for any $a' \in M$
and $b' \in N$ we have
\[
M, a' \models \chi(x) \Leftrightarrow M', f(a') \models \chi(x);
N, b' \models \chi(x) \Leftrightarrow N', g(b') \models \chi(x)
\]
By our assumption, $\varphi(x)$ is invariant with respect to
asimulations and so, by Lemma \ref{L:bisim-asim} we get:
\begin{align}
&M, a \models \varphi(x) \Leftrightarrow M', f(a) \models
\varphi(x)\label{E:m1}\\
&N', g(b) \models \varphi(x)\label{E:m2}
\end{align}
Any standard $x$-translation of an intuitionistic formula is also,
by Corollary \ref{L:c-inv}, invariant with respect to
asimulations. Therefore, we have
\[
IT_\varphi(N',g(b)) = IT_\varphi(N, b) \subseteq IT_\varphi(M,a) =
IT_\varphi(M',f(a)).
\]
But then we have $(N',g(b)) \leq_{\Sigma_\varphi} (M',f(a))$, and
by m-saturation of $M'$, $N'$ and Lemma \ref{L:sat} the relation
$\leq_{\Sigma_\varphi}$ is an
$\langle(N',g(b)),(M',f(a))\rangle$-asimulation. But then by
\eqref{E:m2} and asimulation invariance of  $\varphi(x)$ we get
$M', f(a) \models \varphi(x)$, and further, by \eqref{E:m1} we
conclude that $M, a \models \varphi(x)$. Therefore, $\varphi(x)$
in fact follows from $IC(\varphi(x))$.
\end{proof}

The following theorem is an immediate consequence of Corollary
\ref{L:c-inv} and Theorem \ref{L:main}:
\begin{theorem}\label{L:final}
A formula $\varphi(x)$ is invariant with respect to asimulations
iff it is equivalent to a standard $x$-translation of an
intuitionistic formula.
\end{theorem}

Theorem \ref{L:final} stated above establishes a criterion for the
equivalence of first-order formula to a standard translation of
intuitionistic formula on arbitrary first-order models. But,
unlike in the case of modal propositional logic, some of these
models will not be intended models for intuitionistic logic.
Therefore it would be interesting to look for the criterion of
equivalence of first-order formula to a standard translation of
intuitionistic formula on `intuitionistic' subclass of first-order
models. As the class of intended models of intuitionistic
propositional logic constitutes a first-order definable subclass
of first-order models in general, we can show that such a
criterion is provided by invariance with respect to asimulations
on the models from this subclass using but a slight modification
of our proof for Theorems \ref{L:main} and \ref{L:final}.

To tighten up on terminology, we introduce the following
definitions:
\begin{definition}\label{D:int-mod}
Let $\Sigma' \subseteq \Sigma$. Then $\Sigma'$-model $M$ is
intuitionistic, iff $R^M$ is transitive and reflexive, and it is
true that
\[
\forall(P \in \Sigma')\forall(a,b \in D(M))(aR^Mb \wedge M,a
\models P(x) \Rightarrow M,b \models P(x)).
\]
\end{definition}
The notion of intuitionistic model naturally leads to the
following semantic definitions:
\begin{definition}\label{D:basics}
\begin{enumerate}
\item $\Gamma$ is intuitionistically satisfiable iff $\Gamma$ is
satisfied in some intuitionistic model.\item $\varphi$ is an
intuitionistic consequence of $\Gamma$ $(\Gamma \models_i
\varphi)$ iff $\Gamma \cup \{\,\neg\varphi\,\}$ is
intuitionistically unsatisfiable. \item $\varphi$ is
intuitionistically equivalent to $\psi$ iff both $\psi \models_i
\varphi$ and $\varphi \models_i \psi$.
\end{enumerate}
\end{definition}
For $\Sigma' \subseteq \Sigma$ let $Int(\Sigma')$ be the following
set of formulas
\[
\{\,\forall yR(y,y), \forall yzw((R(y,z) \wedge R(z,w)) \to
R(y,w))\,\} \cup \{\,\forall yz((P(y) \wedge R(y,z)) \to P(z))\mid
P \in \Sigma'\,\}.
\]
It is clear that for any set $\Gamma$ of $\Sigma'$-formulas and
for any $\Sigma'$-formula $\varphi$, $\Gamma$ is
intuitionistically satisfiable iff  $\Gamma \cup Int(\Sigma')$ is
satisfiable, and $\Gamma \models_i \varphi$ iff $\Gamma \cup
Int(\Sigma')\models \varphi$.
\begin{definition}\label{D:int-inv}
A formula $\varphi(x)$ is intuitionistically invariant with
respect to asimulations iff for any $\Sigma'$ such that
$\Sigma_\varphi \subseteq \Sigma' \subseteq \Sigma$, any pointed
intuitionistic $\Sigma'$-models $(M, a)$ and $(N, b)$, if there
exists an $\langle (M, a),(N, b)\rangle$-asimulation $A$ and $M, a
\models \varphi(x)$, then $N, b \models \varphi(x)$.
\end{definition}
\begin{example}
Formula $\exists y(R(x,y) \wedge P(y))$ is not intuitionistically
invariant with respect to asimulations. However, our argument from
Example \ref{ex-inv} does not show this, because models considered
in this example are not intuitionistic. To prove the absence of
intuitionistic invariance with respect to asimulations, consider
two $\{\,R^2,P^1\,\}$-models $M_1$ and $N_1$ such that $D(M_1) =
\{\,a,b,c\,\}$, $R^{M_1} =\{\,(a,a),(a,b), (a,c),(b,b),(c,c)\,\}$,
$P^{M_1} = \{\,c\,\}$, and $D(N_1) = \{\,d,e\,\}$, $R^{N_1}
=\{\,(d,d),(d,e),(e,e)\,\}$, $P^{N_1} = \varnothing$. These are
intuitionistic models. Then binary relation $C =
\{\,(a,d),(b,d),(d,b),(b,e),(e,b)\,\}$ is an $\langle (M_1,a),
(N_1,d)\rangle_k$-asimulation. It remains to note that the formula
under consideration is true at $(M_1,a)$ but not at $(N_1,d)$.
\end{example}
Now for the criterion of equivalence on the restricted class of
intuitionistic models:
\begin{theorem}\label{L:int-main}
Let $\varphi(x)$ be intuitionistically invariant with respect to
asimulations. Then $\varphi(x)$ is intuitionistically equivalent
to a standard $x$-translation of an intuitionistic formula.
\end{theorem}
\begin{proof}
We may assume that $\varphi(x)$ is intuitionistically satisfiable,
otherwise $\varphi(x)$ is intuitionistically equivalent to
$ST(\bot, x)$ and we are done. In what follows we will write
$IntC(\varphi(x))$ for the set of $\Sigma_\varphi$-formulas that
are standard $x$-translations of intuitionistic formulas
intuitionistically following from $\varphi(x)$.

Our strategy will be to show that $IntC(\varphi(x)) \models_i
\varphi(x)$. Once this is done we will conclude that
\[
Int(\Sigma_\varphi) \cup IntC(\varphi(x)) \models \varphi(x).
\]
Then we apply compactness of first-order logic and conclude that
$\varphi(x)$ is equivalent to a finite conjunction
$\psi_1(x)\wedge\ldots \wedge\psi_n(x)$ of formulas from this set.
But it follows then that $\varphi(x)$ is intuitionistically
equivalent to the conjunction of the set $IntC(\varphi(x)) \cap
\{\,\psi_1(x)\ldots, \psi_n(x)\,\}$. In fact, by our choice of
$IntC(\varphi(x))$ we have
\[
\varphi(x) \models_i \bigwedge(IntC(\varphi(x)) \cap
\{\,\psi_1(x)\ldots, \psi_n(x)\,\}),
\]
And by our choice of $\psi_1(x)\ldots, \psi_n(x)$ we have

\[
Int(\Sigma_\varphi) \cup (IntC(\varphi(x)) \cap
\{\,\psi_1(x)\ldots, \psi_n(x)\,\}) \models \varphi(x)
\]
and hence
\[
IntC(\varphi(x)) \cap \{\,\psi_1(x)\ldots, \psi_n(x)\,\} \models_i
\varphi(x).
\]

To show that $IntC(\varphi(x)) \models_i \varphi(x)$, take any
pointed intuitionistic $\Sigma_\varphi$-model $(M, a)$ such that
$M, a \models IntC(\varphi(x))$. Such a model exists, because
$\varphi(x)$ is intuitionistically satisfiable and
$IntC(\varphi(x))$ will be intuitionistically satisfied in any
pointed intuitionistic model satisfying $\varphi(x)$. Then we can
also choose a pointed intuitionistic $\Sigma_\varphi$-model $(N,
b)$ such that $N, b \models \varphi(x)$ and $IT_\varphi(N, b)
\subseteq IT_\varphi(M,a)$.

For suppose otherwise. Then for any pointed intuitionistic
$\Sigma_\varphi$-model $(N, b)$ such that $N, b \models
\varphi(x)$ we can choose an intuitionistic formula $i_{(N, b)}$
such that $ST(i_{(N, b)}, x)$ is a $\Sigma_\varphi$-formula true
at $(N, b)$ but not at $(M, a)$. Then consider the set
\[
S = \{\,\varphi(x)\,\} \cup \{\,\neg ST(i_{(N, b)}, x)\mid N
\text{ is intuitionistic, }N, b \models \varphi(x)\,\}
\]
Let $\{\,\varphi(x), \neg ST(i_{(N_1, b_1)}, x)\ldots , \neg
ST(i_{(N_u, b_u)}, x)\,\}$ be a finite subset of this set. If this
set is  intuitionistically unsatisfiable, then we must have
\[
\varphi(x) \models_i ST(i_{(N_1, b_1)}, x)\vee\ldots \vee
ST(i_{(N_u, b_u)}, x),
\]
but then we will also have
\[
(ST(i_{(N_1, b_1)}, x)\vee\ldots \vee ST(i_{(N_u, b_u)}, x)) \in
IntC(\varphi(x)) \subseteq IT_\varphi(M,a),
\]
and hence $(ST(i_{(N_1, b_1)}, x)\vee\ldots \vee ST(i_{(N_u,
b_u)}, x))$ will be true at $(M, a)$. But then at least one of
$ST(i_{(N_1, b_1)}, x)\ldots ,ST(i_{(N_u, b_u)}, x)$ must also be
true at $(M, a)$, which contradicts the choice of these formulas.
Therefore, every finite subset of $S$ is intuitionistically
satisfiable. But then every finite subset of the set $S \cup
Int(\Sigma_\varphi)$ is satisfiable as well. By compactness of
first-order logic $S \cup Int(\Sigma_\varphi)$ is satisfiable,
hence $S$ is satisfiable intuitionistically. But then take any
pointed intuitionistic $\Sigma_\varphi$-model $(N',b')$ of $S$ and
this will be a model for which we will have both $N', b' \models
ST(i_{(N', b')}, x)$ by choice of $i_{(N', b')}$ and $N', b'
\models \neg ST(i_{(N', b')}, x)$ by the satisfaction of $S$, a
contradiction.

Therefore, for any given pointed intuitionistic
$\Sigma_\varphi$-model $(M, a)$ of $IntC(\varphi(x))$ we can
choose a pointed intuitionistic $\Sigma_\varphi$-model $(N, b)$
such that $N, b \models \varphi(x)$ and $IT_\varphi(N, b)
\subseteq IT_\varphi(M,a)$. Then, reasoning exactly as in the
proof of Theorem \ref{L:main}, we conclude that $M, a \models
\varphi(x)$. Therefore, $\varphi(x)$ in fact intuitionistically
follows from $IntC(\varphi(x))$.
\end{proof}
\begin{theorem}\label{L:int-final}
A formula $\varphi(x)$ is intuitionistically invariant with
respect to asimulations iff it is intuitionistically equivalent to
a standard $x$-translation of an intuitionistic formula.
\end{theorem}
\begin{proof}
From left to right our theorem follows from Theorem
\ref{L:int-main}. In the other direction, assume that $\varphi(x)$
is intuitionistically equivalent to $ST(i,x)$ and assume that for
some $\Sigma'$ such that $\Sigma_\varphi \subseteq \Sigma'
\subseteq \Sigma$, some pointed intuitionistic $\Sigma'$-models
$(M, a)$ and $(N, b)$, and some $\langle (M, a),(N,
b)\rangle$-asimulation $A$ we have $M, a \models \varphi(x)$.
Then, by Corollary \ref{L:c-inv} we have $N, b \models ST(i,x)$,
but since $ST(i,x)$ is intuitionistically equivalent to
$\varphi(x)$ and $N$ is an intuitionistic model, we also have $N,
b \models\varphi(x)$. Therefore, $\varphi(x)$ is
intuitionistically invariant with respect to asimulations.
\end{proof}

\section{Conclusion and further research}\label{S:final}
Theorems \ref{L:param}, \ref{L:final}, and \ref{L:int-final}
proved above show that the general idea of asimulation for
intuitionistic propositional logic is a faithful analogue of the
idea of bisimulation for modal propositional logic in many
important respects.

As for the future research, it is natural to concentrate on
extending the above results onto the level of intuitionistic
predicate logic in order to obtain theorems analogous to Theorem
21 of \cite[p. 124]{vB}. In fact, we already obtained a proof of a
`parametrized' version of such result by extending techniques
employed in the section \ref{S:Param} to cover the predicate case.
We hope to publish this result in some of our future papers.

}

\begin{thebibliography}{99}
\bibitem[Blackburn et~al. 2001]{BRV} Blackburn, P., De Rijke, M., \& Venema,\,Y. (2001). \newblock{\em Modal
Logic.} \newblock Cambridge University Press.
\bibitem[Ebbinghaus et~al. 1984]{EFT}
Ebbinghaus, H.-D., Flum, J., \& Thomas, W.
\newblock{\em Mathematical Logic (1st edition)}. \newblock Springer.
\bibitem[Van Benthem 2010]{vB} Van Benthem, J. \newblock{\em Modal Logic for Open Minds}.
\newblock CSLI Publications.
\end{thebibliography}
\end{document}